\documentclass[reqno,twoside]{amsart}

\usepackage{amsfonts}
\usepackage{amsmath}
\usepackage{amsthm}
\usepackage[utf8]{inputenc}
\usepackage[english]{babel}
\usepackage{epstopdf}
\usepackage{bmpsize}
\usepackage{epsfig}
\usepackage{hyperref}
\usepackage[english]{varioref}
\usepackage{amssymb}
\usepackage{multicol}
\usepackage{dcolumn}
\usepackage{geometry}
\usepackage{fancyhdr}
\usepackage[mathcal]{eucal}
\usepackage{mathrsfs}
\usepackage{color, colortbl}
\usepackage{microtype}
\usepackage{longtable}
\usepackage[toc,page]{appendix}
\usepackage{bm}
\usepackage{pifont}
\usepackage{fleqn}
\usepackage{graphicx}
\usepackage{txfonts}
\usepackage{fullpage}
\usepackage{tikz}
\usepackage{pgfplots}
\usepackage{pgfplotstable}

\usetikzlibrary{external}

\DisableLigatures{encoding = *, family = * }

\newcommand{\D}{\displaystyle}
\newcommand{\norm}[2]{\left\|#1\right\|_{#2}}

\newcommand{\R}{\mathbb{R}}

\newcommand{\intd}{\displaystyle\int \kern -7pt \int}

\newtheorem{theorem}{\bf Theorem}[section]
\newtheorem{remark}{\bf Remark}[section]
\newtheorem{proposition}{\bf Proposition}[section]

\newtheorem{lemma}{\bf Lemma}[section]

\title{Null controllability of linear and semilinear nonlocal heat equations with integral kernel}

\author{Umberto Biccari\textsuperscript{1}}  
\address{\textsuperscript{1,2}\,DeustoTech, University of Deusto, 48007 Bilbao, Basque Country, Spain.}
\address{\textsuperscript{1,2}\,Facultad de Ingenier\'ia, Universidad de Deusto, Avenida de las Universidades 24, 48007 Bilbao, Basque Country, Spain.}
\email{umberto.biccari@deusto.es, u.biccari@gmail.com, victor.santamaria@deusto.es}

\author{V\'ictor Hern\'andez-Santamar\'ia\textsuperscript{2}}

\keywords{Null controllability, heat equation, nonlocal terms, Carleman inequalities, linear and semilinear systems}
\subjclass[2010]{35K58,\,93B05,\,93B07,\,93C20}

\begin{document}           

\begin{abstract}
We consider a linear nonlocal heat equation in a bounded domain $\Omega\subset\R^d$ with Dirichlet boundary conditions. The non-locality is given by the presence of an integral kernel. We analyze the problem of controllability when the control acts on an open subset of the domain. It is by now known that the system is null-controllable when the kernel is time-independent and analytic or, in the one-dimensional case, in separated variables. In this paper, we relax this assumption and we extend the result to a more general class of kernels. Moreover, we get explicit estimates on the cost of null-controllability that allow us to extend the result to some semilinear models.
\end{abstract}

\maketitle

\section{Introduction and main results}\label{intro_sec}
Let $\Omega$ be a bounded domain of $\R^d$ with boundary of class $C^2$. Given $T>0$, we set $Q:=\Omega\times (0,T)$ and $\Sigma:=\partial\Omega\times (0,T)$. Let $K=K(x,\theta,t)\in L^\infty(\Omega\times\Omega\times(0,T))$.  We consider the following linear parabolic equation involving a nonlocal space term. 
\begin{align}\label{e1s1}
	\begin{cases}
		\D y_t - \Delta y + \int_\Omega K(x,\theta,t)y(\theta,t)\,d\theta = v\mathbf{1}_{\mathcal O}, & (x,t)\in Q
		\\
		y = 0, & (x,t)\in\Sigma
		\\
		y(x,0) = y_0(x), & x\in\Omega.
	\end{cases}
\end{align}

In (\ref{e1s1}), $y=y(x,t)$ is the state and $v=v(x,t)$ is the control. The latter acts on the system through the non-empty open subset ${\mathcal O}\subset\Omega$. Here, $\mathbf{1}_{\mathcal O}$ denotes the characteristic function of $\mathcal O$. 

We assume that $y_0\in L^2(\Omega)$ and $v\in L^2(\mathcal O\times(0,T))$, so that system \eqref{e1s1} admits a unique solution $y$ in the class
\begin{align}\label{y_reg}
	y\in L^2(0,T;H^{1}_{0}(\Omega))\cap H^1(0,T;H^{-1}(\Omega)),
\end{align}
which satisfies classical energy estimates. Actually, this remains true also if $v\mathbf{1}_{\mathcal O}$ is replaced by a general right-hand side $f\in L^2(0,T;H^{-1}(\Omega))$.

We are interested in proving the null controllability of the problem under analysis. In other words, we want to show that there exists a control function $v\in L^2(\mathcal O\times(0,T))$ such that the corresponding solution $y$ to \eqref{e1s1} satisfies $y(x,T)=0$ for all $T>0$.

Moreover, as mentioned in \cite{fernandez2016null}, the study of the controllability of \eqref{e1s1} is motivated by many relevant applications from physics and biology. See, for instance, \cite[Section 7.9.2]{okubo2013diffusion}, where this kind of equations is used in the study of group dynamics, for modeling the possibility of interactions between individuals that are separated in space. 

It is well known that system \eqref{e1s1} is null controllable at least in two cases.
\begin{itemize}
	\item When the kernel is time-independent and analytic, one can exploit unique continuation properties and use compactness-uniqueness arguments (\cite{fernandez2016null}). In this framework, also coupled systems have been recently treated in \cite{lissy2018internal}.
	\item When the problem is one-dimensional and the kernel is time-independent and in separated variables, the controllability follows employing spectral analysis techniques (\cite{micu2018local}).
\end{itemize}

In the present paper, by means of a Carleman approach, we are able to extend the above mentioned results by considering a problem in any space dimension and by weakening the assumptions on the kernel. In particular, we will only need ${K=K(x,\theta,t)}$ to be bounded and to have an exponential decay at the extrema of the time interval $[0,T]$. This is summarized in the following condition:
\begin{align}\label{K_est_gen}
	\mathcal{K}=:\sup_{(x,t)\in\overline{Q}} \exp\left(\frac{\sigma^-}{t(T-t)}\right)\int_{\Omega} |K(x,\theta,t)|\,d\theta < +\infty, \tag{$\mathcal{H}$}
\end{align}
where the constant $\sigma^-$ depends only on $\Omega$ and $\mathcal{O}$ and is related to the Carleman weight (see \eqref{notation_alpha}). Our first main result will then be the following.

\begin{theorem}\label{control_thm} 
Suppose that the kernel $K=K(x,\theta,t)\in L^\infty(\Omega\times\Omega\times(0,T))$ satisfies \eqref{K_est_gen}. Then, given $y_0\in L^2(\Omega)$ and $T>0$, there exists a control function $v\in L^2(\mathcal{O}\times (0,T))$ such that the corresponding solution to \eqref{e1s1} satisfies $y(x,T)=0$.
\end{theorem}

Moreover, it is well known that this null controllability property is equivalent to the observability of the following adjoint system
\begin{align}\label{e11s1}
	\begin{cases}
		\D-\varphi_t-\Delta\varphi + \int_\Omega K(\theta,x,t)\varphi(\theta,t)\,d\theta = 0, & (x,t)\in Q
		\\
		\varphi = 0, & (x,t)\in \Sigma
		\\
		\varphi(x,T) = \varphi_T(x), & x\in \Omega.
	\end{cases}
\end{align}
Therefore, in order to prove Theorem \ref{control_thm}, we are going to show that the following result holds.

\begin{theorem}\label{th2s1} {\bf(Observability estimate).}
	For any solution of \eqref{e11s1} and for any kernel $K$ satisfying \eqref{K_est_gen} there exist positive constants $C_1$ and $C_2$, depending only on $\Omega$ and $\mathcal O$, such that 
	\begin{align}\label{e12s1}
		\norm{\varphi(x,0)}{L^2(\Omega)}^2 \leq \frac{C_1}{T} \exp\left[C_2\left(1 + \mathcal{K}^{\frac 23} +\frac
	1T \right)\right] \intd_{{ \mathcal O}\times(0,T)}|\varphi|^{2}\,dx\,dt\,.
	\end{align}
\end{theorem}
Once \eqref{e12s1} is known, Theorem \ref{control_thm} can be proved employing a classical arguments. 

The proof of the above inequality, in turn, relies on a global Carleman estimate, in which we pay special attention to the presence of the nonlocal term. Moreover, we mention that Carleman inequalities for equations similar to \eqref{e1s1} have been obtained in the context of the study of ill-posed problems (see, e.g., \cite{lorenzi2011two}). 

Since equation \eqref{e1s1} turns out to be null-controllable in any time $T>0$, a natural issue to be analyzed is the {\it cost of null controllability} or, more precisely, the cost of driving the solution to \eqref{e1s1} from $y_0$ to zero. With this purpose, let us recall that this cost is measured by the following quantity:
\begin{align}\label{cost}
	\mathcal C(y_0)=\D\inf_{v\in{\mathcal O\times(0,T)}}\norm{v}{L^2(\mathcal O\times(0,T))}\,.
\end{align}

It is classically known that the cost of null controllability for a heat-like equation blows-up as $T\to 0^+$. We will see later in this work that this is the case also for our nonlocal problem.

Finally, observe that hypothesis \eqref{K_est_gen} implies that the kernel $K$ has to vanish exponentially as $t$ goes to $0^+$ and to $T^-$. Nevertheless, following the classical approach of \cite{fernandez2004local} (see also \cite{tao2016null}), it is possible to remove the decay assumption at $t=0$, but this at the price of losing any information on the controllability cost. In fact, in this case we shall argue by a fixed point procedure, implying that we do not have a constructive method to build the control. We will discuss this fact with more details later in this work.  

The rest of the paper is organized as follows. In Section \ref{obs_sec}, we prove in detail the observability estimate in Theorem \ref{th2s1}, from which the proof of Theorem \ref{control_thm} follows immediately. Moreover, we will present an abridged discussion concerning the null controllability cost associated to our problem. In Section \ref{decay_sec}, we show that it is possible to remove the decay assumption for the kernel at $t=0$, but with the drawback of losing an explicit constant in the observability inequality and, consequently, the possibility of analyzing the cost of null controllability. In Section \ref{nl_sec}, we will briefly consider the extension of our result to the semilinear case. Finally, Section \ref{comments_sec} will be devoted to some additional comments on the necessity of the hypothesis \eqref{K_est_gen}.

\section{Proof of the observability inequality and of the controllability result}\label{obs_sec}
The observability inequality for the solutions to the adjoint system \eqref{e11s1} is a consequence of a suitable global Carleman estimate. In the sequel, $C$ stands for a generic positive constant only depending on $\Omega$ and ${\mathcal O}$, whose value can change from line to line.

According to \cite[Lemma 1.1]{fursikov1996controllability}, we have the following. 

\begin{lemma}
Let $\mathcal O\subset\subset \Omega$ be a nonempty open set. Then, there exists $\eta^0\in C^2(\overline{\Omega})$ such that $\eta^0>0$ in $\Omega$, $\eta^0=0$ on $\partial\Omega$ and $|\nabla\eta^0|> 0$ in $\overline{\Omega\setminus{ \mathcal O}}$.
\end{lemma}

Now, for a parameter $\lambda>0$, we define 
\begin{align*}
	\sigma(x):= e^{4\lambda\norm{\eta^0}{\infty}}-e^{\lambda\left(2\norm{\eta^0}{\infty}+\eta^0(x)\right)},
\end{align*}
and we introduce the weight functions
\begin{align}\label{weight}
	\alpha(x,t):=\frac{\sigma(x)}{t(T-t)}, \;\;\;\;\;\;\;\;\; \xi(x,t):=\frac{e^{\lambda\left(2\norm{\eta^0}{\infty}+\eta^0(x)\right)}}{t(T-t)}.
\end{align}
Moreover, in what follows we will use the notation 
\begin{align}\label{notation_alpha}
	\sigma^+:= \max_{x\in\overline{\Omega}}\sigma(x)=e^{4\lambda\norm{\eta^0}{\infty}}-e^{2\lambda\norm{\eta^0}{\infty}}, \;\;\;\;\;\;\;\;\; \sigma^-:= \min_{x\in\overline{\Omega}}\sigma(x)=e^{4\lambda\norm{\eta^0}{\infty}}-e^{3\lambda\norm{\eta^0}{\infty}}, 
\end{align}
and we introduce the following quantity to abridge the computations
\begin{align*}
	\mathcal I(\cdot):= s\lambda^2\intd_Q e^{-2s\alpha}\xi|\nabla \cdot|^2\,dx\,dt + s^3\lambda^4\intd_Q e^{-2s\alpha}\xi^3|\cdot|^2\,dx\,dt.
\end{align*}	
Finally, for our further results we are going to use the estimate
\begin{align}\label{xi_est}
	\xi(t)^{-\nu}\leq CT^{2\nu},\;\;\;\forall\nu>0.
\end{align}

Then, \cite[Lemma 1.3]{fernandez2006global} gives the following.

\begin{proposition}\label{carleman z_prop} 
There exist positive constants $C$ and $s_1$ such that, for all $s\geq s_1$, $\lambda\geq C$, $F\in L^2(Q)$ and $z_T\in L^2(\Omega)$, the solution $z$ to 
\begin{align}\label{syst_z}
	\begin{cases}
		z_t+\Delta z = F, & (x,t)\in Q
		\\
		z=0, & (x,t)\in\Sigma
		\\
		z(x,T) = z_T(x), & x\in \Omega
	\end{cases}
\end{align}
satisfies
\begin{align}\label{carleman z}
	\mathcal{I}(z) \leq C \left[s^3\lambda^4\intd_{\mathcal O\times(0,T)} e^{-2s\alpha}\xi^3 |z|^2\,dx\,dt + \intd_Q e^{-2s\alpha}|F|^2\,dx\,dt\right]
\end{align}
Moreover, $s_1$ is of the form 
\begin{align}\label{e6s2}
	s_1 = \varrho_1\left(T+T^2\right).
\end{align}
where $\varrho_1$ is a positive constant that only depends on $\Omega$ and $\mathcal O$.
\end{proposition}

Furthermore, in what follows we will need the following technical result, whose proof is inspired by \cite[Lemma 6.1]{montoya2017robust}. 

\begin{proposition}\label{alpha_min_max_prop}
For any fixed $\lambda>0$ and $s>1$ it holds
\begin{align}\label{alpha_min_max}
	\exp\left(-\frac{(1+s)\sigma^-}{t(T-t)}\right)< \exp\left(-\frac{s\sigma^+}{t(T-t)}\right),
\end{align}
where $\sigma^-$ and $\sigma^+$ have bee introduced in \eqref{alpha_min_max}.
\end{proposition}
\begin{proof}
First of all, from the definitions of $\sigma^-$ and $\sigma^+$ (see \eqref{notation_alpha}), we have $\sigma^-=F(\lambda)\sigma^+$, with
\begin{align*}
	F(\lambda):=\frac{e^{2\lambda\norm{\eta^0}{\infty}}-e^{\lambda\norm{\eta^0}{\infty}}}{e^{2\lambda\norm{\eta^0}{\infty}}-1}.
\end{align*}

It is straightforward to check that $F(\lambda)$ is a monotone increasing function, verifying $\lim_{\lambda\to+\infty} F(\lambda) = 1$ and $\lim_{\lambda\to 0^+} F(\lambda) = 1/2$ (see Figure \ref{F_lambda_fig}). 
\begin{figure}[h]
	\centering 
	\begin{tikzpicture}[scale=0.9]
	\pgfplotsset{settings/.style={legend cell align = {left},
			legend pos =  south east,
			legend style = {draw = none}}}
	\begin{axis}[settings,xmin=0,xmax=10, ytick={0.5,1}]
	
	\addplot[domain=0.01:10, samples=500,thick,color=blue]{(exp(2*x)-exp(x))/(exp(2*x)-1)};\addlegendentry{$\;\;F(\lambda)$};
	\addplot[domain=0.01:10, samples=500,dashed]{1};\addlegendentry{\;\;asymptote $F(\lambda)=1$};
		\end{axis}
	\end{tikzpicture}\caption{Profile of the function $F(\lambda)$ for $\lambda\geq 0$.}\label{F_lambda_fig}
\end{figure}
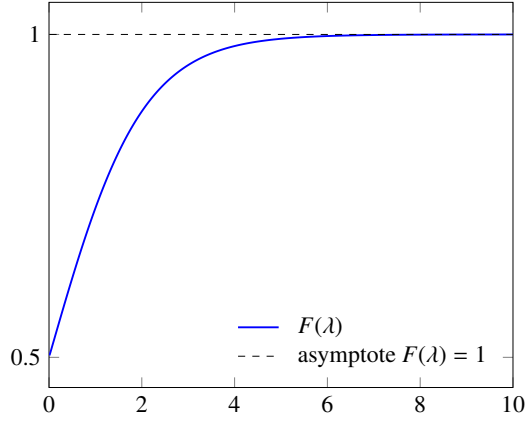
Moreover, since $s>1$ we have 
\begin{align}\label{F_lambda_est}
	(1+s)F(\lambda)> 2F(\lambda)> 1.
\end{align}
Then, multiplying both sides of \eqref{F_lambda_est} by $\sigma^+$, and since $[t(T-t)]^{-1}>0$, we immediately have 
\begin{align*}
	\sigma^+[t(T-t)]^{-1}\Big(1-(1+s)F(\lambda)\Big)< 0.
\end{align*}
Hence
\begin{align*}
	\exp\left(\frac{\sigma^+}{t(T-t)}\Big(1-(1+s)F(\lambda)\Big)\right)< 1
\end{align*}
and we can conclude
\begin{align*}
	\exp\left(\frac{s\sigma^+}{t(T-t)}\right) < \exp\left(\frac{(1+s)F(\lambda)\sigma^+}{t(T-t)}\right) = \exp\left(\frac{(1+s)\sigma^-}{t(T-t)}\right).
\end{align*}
From this, \eqref{alpha_min_max} follows immediately.
\end{proof}

Proposition \ref{carleman z_prop} can now be applied to the solutions to \eqref{e11s1}, and we obtain the following Carleman estimate.

\begin{proposition}\label{carleman_phi_prop} 
Let $\varphi^T\in L^2(\Omega)$ and assume that the kernel $K$ satisfies \eqref{K_est_gen}. Then, there exist positive constants $C$, $\lambda_0$ and $\varrho_2$, only depending on $\Omega$ and $\mathcal O$, such that the solution $\varphi$ to \eqref{e11s1} corresponding to the initial datum $\varphi^T$ satisfies  
\begin{align}\label{e5s2}
	\mathcal{I}(\varphi) \leq C s^3\lambda^4\intd_{\mathcal O\times(0,T)} e^{-2s\alpha}\xi^3 |\varphi|^2\,dx\,dt, 
\end{align}
for any $\lambda\geq \lambda_0$ and any $s\geq \varrho_2\left(T + T^2 + \mathcal{K}^{\frac 23}T^2\right)$.
\end{proposition}
\begin{proof}
We begin applying \eqref{carleman z} to $\varphi$, obtaining, for any $\lambda\geq C$ and any $s\geq \varrho_1\left(T + T^2\right)$ 
\begin{align}\label{carleman_prel}
	\mathcal{I}(\varphi) \leq C \left[s^3\lambda^4\intd_{\mathcal O\times(0,T)} e^{-2s\alpha}\xi^3 |\varphi|^2\,dx\,dt + \intd_Q e^{-2s\alpha}\left|\int_{\Omega} K(\theta,x,t)\varphi(\theta,t)\,d\theta\,\right|^2\,dx\,dt\right].
\end{align}

We are now going to deal with the second term on the right-hand side of the previous estimate. To this end, we set the parameter  $\lambda$ to a fixed value large enough. We have
\begin{align}\label{ker_int}
	\left|\int_\Omega K(\theta,x,t)\varphi(\theta,t)\,d\theta\,\right| & = \left|\int_{\Omega} e^{\frac{\sigma^-}{t(T-t)}}K(\theta,x,t)e^{-\frac{\sigma^-}{t(T-t)}}\varphi(\theta,t)\,d\theta\,\right| \notag
	\\	
	&\leq \left[\left(\int_\Omega e^{\frac{2\sigma^-}{t(T-t)}}|K(\theta,x,t)|^2\,d\theta\right)\left(\int_\Omega e^{-\frac{2\sigma^-}{t(T-t)}}|\varphi(x,\theta)|^2\,d\theta\right)\right]^{\frac 12}.
\end{align}

Notice that, since $\lambda$ has been fixed, $\sigma^-$ (and therefore $\sigma^+$) is a constant depending only on $\Omega$ and $\mathcal O$. Now, replacing \eqref{ker_int} into \eqref{carleman_prel} we get
\begin{align*}
	\mathcal{I}(\varphi) \leq C \left[s^3\lambda^4\intd_{\mathcal O\times(0,T)} e^{-2s\alpha}\xi^3 |\varphi|^2\,dx\,dt + \mathcal{K}^2\intd_Q e^{-2s\alpha(x,t)}\left(\int_{\Omega} e^{-\frac{2\sigma^-}{t(T-t)}}|\varphi(\theta,t)|^2\,d\theta\,\right)\,dx\,dt\right].
\end{align*}

Let us now focus on the second term in the right-hand side of the above inequality. Using Fubini's Theorem we get
\begin{align*}
	\intd_Q e^{-2s\alpha(x,t)} \left(\int_{\Omega} e^{-\frac{2\sigma^-}{t(T-t)}}|\varphi(\theta,t)|^2\,d\theta\,\right)\,dx\,dt = \intd_Q e^{-\frac{2\sigma^-}{t(T-t)}}|\varphi(\theta,t)|^2\left(\int_\Omega e^{-2s\alpha(x,t)}\,dx\right)\,d\theta\,dt.
\end{align*}
Notice that, according to \eqref{notation_alpha}, we have
\begin{align*}
	\int_\Omega e^{-2s\alpha(x,t)}\,dx \leq |\Omega| e^{-\frac{2s\sigma^-}{t(T-t)}},
\end{align*}
where $|\Omega|$ stands for the measure of $\Omega$. Hence, we can compute
\begin{align*}
	\intd_Q e^{-\frac{2\sigma^-}{t(T-t)}}|\varphi(\theta,t)|^2\left(\int_\Omega e^{-2s\alpha(x,t)}\,dx\right)\,d\theta\,dt &\leq C \intd_Q e^{-\frac{2(1+s)\sigma^-}{t(T-t)}}|\varphi(\theta,t)|^2\,d\theta\,dt
	\\
	&\leq C\intd_Q e^{-\frac{2s\sigma^+}{t(T-t)}}|\varphi(\theta,t)|^2\,d\theta\,dt 
	\\
	&\leq C \intd_Q e^{-2s\alpha(\theta,t)}|\varphi(\theta,t)|^2\,d\theta\,dt,
\end{align*}	
were we have used Proposition \ref{alpha_min_max_prop} and the definition of $\sigma^+$. Putting all together, we get
\begin{align*}
	\mathcal{I}(\varphi) \leq C \left[s^3\lambda^4\intd_{\mathcal O\times(0,T)} e^{-2s\alpha}\xi^3 |\varphi|^2\,dx\,dt + \mathcal{K}^2\intd_Q e^{-2s\alpha(x,t)}|\varphi(x,t)|^2\,dx\,dt\right].
\end{align*}
Recalling now the definition of $\mathcal{I}(\varphi)$, we then find the following estimate
\begin{align}\label{carleman_prel3}
	s\lambda^2\intd_Q e^{-2s\alpha}\xi & |\nabla \varphi|^2\,dx\,dt + s^3\lambda^4\intd_Q e^{-2s\alpha}\xi^3|\varphi|^2\,dx\,dt \notag
	\\
	&- \mathcal{K}^2\intd_Q e^{-2s\alpha}|\varphi|^2\,dx\,dt \leq C s^3\lambda^4\intd_{\mathcal O\times(0,T)} e^{-2s\alpha}\xi^3 |\varphi|^2\,dx\,dt.
\end{align}
Therefore, thanks to \eqref{xi_est}, we finally obtain
\begin{align*}
	s\lambda^2\intd_Q e^{-2s\alpha}\xi |\nabla \varphi|^2\,dx\,dt + s^3\lambda^4\intd_Q e^{-2s\alpha}\xi^3|\varphi|^2\,dx\,dt \leq C s^3\lambda^4\intd_{\mathcal O\times(0,T)} e^{-2s\alpha}\xi^3 |\varphi|^2\,dx\,dt,
\end{align*}
for all $s>C\mathcal{K}^{\frac 23}T^2$. This, together with \eqref{e6s2} concludes the proof.
\end{proof}
\begin{remark}\label{rem_lambda}
The hypothesis \eqref{K_est_gen} plays a fundamental role in the previous proof. Notice that, according to this assumption, the kernel $K$, as a function of $t$, should behave like 
\begin{align*}
	K(\cdot,\cdot,t)\sim e^{-\frac{C(\Omega,\mathcal O)}{t(T-t)}},
\end{align*}
i.e. it should decay exponentially as $t$ goes to $0^+$ and $T^-$. We stress that this constant $C$ does not depend on the parameter $s$ and, therefore, in the last step of the above proof there is no obstruction in using this parameter for absorbing the third term on the left hand side of the estimate \eqref{carleman_prel3}.
	
The assumption \eqref{K_est_gen} may appear as a quite strong restriction on the admissible kernels. Notwithstanding, it is instead a natural one, since the only thing that we are asking is integrability of $K$ with respect to the Carleman weight. Moreover, we mention that this is the minimum decay that we shall ask for the kernel. Indeed, following the proof of Proposition \ref{carleman_phi_prop} it is clear that imposing a weaker decay (e.g. polynomial) it is not sufficient to obtain the desired inequality.
\end{remark}
\begin{proof}[Proof of Theorem \ref{th2s1}]

First of all, from \eqref{e5s2} we clearly have
\begin{align*}
	s^3\intd_Q e^{-2s\alpha}\xi^3|\varphi|^2\,dx\,dt \leq C s^3\intd_{\mathcal O\times(0,T)} e^{-2s\alpha}\xi^3 |\varphi|^2\,dx\,dt.
\end{align*}
Moreover, due to the definition of the weight function $\alpha$ (see \eqref{weight}) we have the following two estimates:
\begin{enumerate}
	\item[1.] $s^3 e^{-2s\alpha}\xi^3 \leq Cs^3T^{-6}e^{-\frac{Cs}{T^2}}\leq C(T)$
	
	\item [2.] $s^3 e^{-2s\alpha}\xi^3 \geq Ce^{-\frac{Cs}{T^2}}$, if $t\in\Big[\frac T4,\frac 34 T\Big]$
\end{enumerate}
if we choose $s\geq CT^2$. Therefore, we obtain
\begin{align}\label{obs_prel}
	\int_{\frac T4}^{\frac 34 T}\!\!\!\int_\Omega |\varphi|^2\,dx\,dt \leq C e^{\frac {Cs}{T^2}}\intd_{\mathcal O\times(0,T)} |\varphi|^2\,dx\,dt.
\end{align}

Furthermore, due to classical energy estimates, it is easy to check that $t\mapsto \norm{\varphi(t)}{L^2(\Omega)}$ is an increasing function. Hence,
\begin{align*}
	\int_{\frac T4}^{\frac 34 T}\!\!\!\int_\Omega |\varphi(x,t)|^2\,dx\,dt \geq \int_{\frac T4}^{\frac 34 T}\!\!\!\int_\Omega |\varphi(x,0)|^2\,dx\,dt = \frac T2 \norm{\varphi(x,0)}{L^2(\Omega)}^2,
\end{align*}
and from this last estimate and \eqref{obs_prel} we finally obtain \eqref{e12s1}.
\end{proof}

Once we have the observability inequality, the control $v$ driving the solution $y$ to \eqref{e1s1} from the initial datum $y_0$ to zero can be identified as $v=\left.\varphi\right|_{\mathcal O}$, where $\varphi$ is the solution to the adjoint equation \eqref{e11s1} corresponding to an initial datum $\varphi_T\in L^2(\Omega)$ which is the unique minimizer of the functional
\begin{align}\label{functional_J}
	J\left(\varphi^T\right):= \frac 12\intd_{\mathcal O\times (0,T)} |\varphi|^2\,dxdt + \int_\Omega y_0(x)\varphi(x,0)\,dx.
\end{align}

In more detail, inequality \eqref{e12s1} ensures the coercivity of the above functional. The proof of this fact being classical, we will omit it here. 

Let us conclude this Section with a brief discussion on cost of null controllability for problem \eqref{e1s1}. We recall from \eqref{cost} that this quantity is defined as 
\begin{align*}
	\mathcal C(y_0)=\D\inf_{v\in{\mathcal O\times(0,T)}}\norm{v}{L^2(\mathcal O\times(0,T))}\,.
\end{align*}

On the other hand, it is well-known that this controllability cost may also be characterized in terms of the constant in the observability inequality \eqref{e12s1}. In more detail, we have 
\begin{align*}
	\mathcal C(y_0)=\inf_{C>0}\left\{\norm{\varphi(x,0)}{L^2(\Omega)} \leq C^2 \intd_{{ \mathcal O}\times(0,T)}|\varphi|^{2}\,dx\,dt\,\right\}.
\end{align*}
Applied to our problem, this gives the estimate
\begin{align}\label{cost_est}
	\mathcal C(y_0)\leq\exp\left[C\left(1+\frac{1}{T}+\mathcal{K}^{\frac 23}\right)\right].
\end{align}

In particular, as it is natural to expect for a heat-like equation, the null controllability cost blows-up as $T\to 0^+$. This means that, even if equation \eqref{e1s1} is null controllable for any time $T>0$, the cost of this process is growing exponentially as the time interval $(0,T)$ shrinks. Lastly, observe that $\mathcal K^{\frac 23}$ in \eqref{cost_est} is in accordance with the classical results on the controllability of heat equations with bounded potentials (see, e.g., \cite{fernandez2000cost}).

\section{Removing the assumption on the decay of the kernel in $t=0$}\label{decay_sec}

We are interested in showing that the assumption \eqref{K_est_gen} on the decay in time of the kernel $K$ as $t$ goes to $0^+$ and $T^-$ can be substituted by the following one, which does not requires any decay at $t=0$:
\begin{align}\label{K_est_weak}
	\mathcal{M}:=\sup_{(x,t)\in\overline{Q}}\exp\left(\frac{\mathcal{B}}{T-t}\right)\int_{\Omega} |K(x,\theta,t)|\,d\theta <+\infty.
\end{align}
Let us now introduce the new weights $\beta$ and $\gamma$ defined as
\begin{align*}
	\beta(x,t):=\frac{e^{4\lambda\norm{\eta^0}{\infty}}-e^{\lambda\left(2\norm{\eta^0}{\infty}+\eta^0(x)\right)}}{\ell(t)}, \;\;\; \gamma(x,t):=\frac{e^{\lambda\left(2\norm{\eta^0}{\infty}+\eta^0(x)\right)}}{\ell(t)},
\end{align*}
with
\begin{align*}
	\ell(t):=\begin{cases}
		\D T^2/4, & t\in\left[0,T/2\right]
		\\
		\D t(T-t), & t\in\left[T/2, T\right],
	\end{cases}
\end{align*}
and the parameters $s$ and $\lambda$ are fixed and taken as in Proposition \ref{carleman z_prop}. Then, we can state the following refined version of the Carleman inequality \eqref{carleman z}.

\begin{proposition}
There exist a positive constants $C$, depending on $T$, $s$ and $\lambda$, such that, for all $F\in L^2(Q)$ and $z_T\in L^2(\Omega)$, the solution $z$ to \eqref{syst_z} satisfies 
\begin{align}\label{carleman_z_ref}
	\norm{z(x,0)}{L^2(\Omega)}^2 + \intd_Q e^{-2s\beta}\gamma^3|z|^2\,dxdt \leq C \left[\intd_{\mathcal O\times(0,T)} e^{-2s\beta}\gamma^3 |z|^2\,dx\,dt + \intd_Q e^{-2s\beta}|F|^2\,dx\,dt\right].
\end{align}
\end{proposition}

The proof of this Proposition is standard. It combines energy estimates and the fact that $\beta\leq\alpha$ in $Q$ (see, e.g., \cite{fernandez2004local}). Furthermore, using \eqref{carleman_z_ref} and the classical approach presented in several works (\cite{fernandez2004local,fursikov1996controllability,gueye2013insensitizing,tao2016null}), it is possible to obtain the following result.

\begin{proposition}\label{control_prop_F}
Let $T>0$ and $e^{s\beta}F\in L^2(Q)$. Then, for any $y_0\in L^2(\Omega)$ there exists a control function $v\in L^2(\mathcal O\times(0,T))$ such that the associated solution to \eqref{e1s1} is in the space
\begin{align*}
	\mathcal E:=\left\{y\,:\,e^{s\beta}y\in L^2(Q)\right\}.
\end{align*}
Moreover, there exists a positive constant $C=C(T,s,\lambda)$ such that it holds the estimate
\begin{align}\label{energy_weight}
	\intd_{\mathcal O\times(0,T)} e^{2s\beta}\gamma^{-3} |v|^2\,dx\,dt + \intd_Q e^{2s\beta}|y|^2\,dx\,dt 
	 \leq C \left(\norm{y_0}{L^2(\Omega)}^2 + \norm{e^{s\beta}F}{L^2(Q)}^2\right).
\end{align}
\end{proposition} 
Notice that, $y$ being in the space $\mathcal E$, in particular we have 
\begin{align*}
	\intd_Q e^{2s\beta}|y|^2\,dx\,dt < +\infty.
\end{align*}

Since the weight $\beta$ blows-up as $t\to T^-$, the boundedness of the above integral yields $y(x,T)=0$. As a consequence, we then have the following controllability result.
\begin{proposition}
Let $T>0$ and $e^{s\beta}F\in L^2(Q)$. Then, for any $y_0\in L^2(\Omega)$ there exists a control function $v\in L^2(\mathcal O\times(0,T))$ such that the associated solution to \eqref{e1s1} satisfies $y(x,T)=0$.
\end{proposition} 
 
The above discussion can now be applied to problem \eqref{e1s1}, and we have the following result.
\begin{theorem}\label{control_thm2} 
Let $T>0$ and assume that $K$ satisfies \eqref{K_est_weak}. Then, for any $y_0\in L^2(\Omega)$, there exists a control function $v\in L^2(\mathcal O\times(0,T))$ such that the associated solution $y$ to \eqref{e1s1} satisfies $y(x,T) = 0$.	
\end{theorem}	
\begin{proof}
For our proof we are going to employ a fixed point strategy. For $R>0$, we define 
\begin{align*}
	\mathcal E_R:=\left\{ w\in\mathcal E\,:\, \norm{e^{s\beta}w}{L^2(Q)}\leq R\right\},
\end{align*}
which is a bounded, closed and convex subset of $L^2(Q)$. For any $w\in\mathcal E_R$, let us consider the control problem 
\begin{align}\label{control_w}
	\begin{cases}
		\D y_t - \Delta y + \int_\Omega K(x,\theta,t)w(\theta,t)\,d\theta = v\mathbf{1}_{\mathcal O}, & (x,t)\in Q
		\\
		y = 0, & (x,t)\in\Sigma
		\\
		y(x,0) = y_0(x), & x\in\Omega.
	\end{cases}
\end{align}	
Notice that \eqref{control_w} is different from the original system \eqref{e1s1}, since it is linear in the $y$ variable. Now, from hypothesis \eqref{K_est_weak} we have that
\begin{align*}
	\intd_Q\left(e^{s\beta}\int_\Omega K(x,\theta,t)w(\theta,t)\,d\theta\right)^2\,dxdt \leq \mathcal M^2\intd_Q e^{2s\beta}w^2e^{-2s\beta}\,dxdt \leq \mathcal M^2R^2.
\end{align*}

Therefore, from Proposition \ref{control_prop_F} we have that \eqref{control_w} is null controllable, i.e. for any $y_0\in L^2(\Omega)$, there exists a control function $v\in L^2(\mathcal O\times(0,T))$ such that the associated solution $y$ to \eqref{control_w} satisfies $y(x,T) = 0$. 

In order to conclude our proof and obtain the same controllability result for $w=y$, we shall apply Kakutani's fixed point theorem (see \cite[Theorem 2.3]{fernandez2006global}, \cite{kakutani1941generalization}). For any $w\in \mathcal E_R$, we define the multivalued map $\Lambda: \mathcal E_R\mapsto 2^{\mathcal E}$ such that
\begin{align*}
	\Lambda(w) = \left\{y\,:\,y\in\mathcal E \textrm{ and there exists } v \textrm{ such that }	\intd_{\mathcal O\times(0,T)} e^{2s\beta}\gamma^{-3} |v|^2\,dx\,dt \leq C \left(R^2 + \norm{y_0}{L^2(\Omega)}^2\right). \right\}
\end{align*} 	

It is easy to check that $\Lambda(w)$ is a nonempty, closed and convex subset of $L^2(Q)$. Moreover, by \eqref{K_est_weak} and \eqref{energy_weight} and arguing as before we have
\begin{align}\label{estimate}
	\intd_{\mathcal O\times(0,T)} e^{2s\beta}\gamma^{-3} |v|^2\,dx\,dt & + \intd_Q e^{2s\beta}|y|^2\,dx\,dt \notag 
	\\
	&\leq C \left[\norm{y_0}{L^2(\Omega)}^2 + \intd_Qe^{2s\beta}\left(\int_\Omega K(x,\theta,t)y(\theta,t)\,d\theta\right)^2\,dxdt\right] \notag 
	\\
	&\leq C\left(\mathcal M^2R^2 + \norm{y_0}{L^2(\Omega)}^2\right) \leq CR^2,
\end{align}
for $R$ large enough. Hence, up to a multiplicative constant we have $\Lambda(\mathcal E_R)\subset \mathcal E_R$. 

Let $\left\{w_k\right\}$ be a sequence in $\mathcal E_R$. Then the corresponding solutions $\left\{y_k\right\}$ are bounded in $L^2(0,T;H^1_0(\Omega))\cap H^1(0,T;H^{-1}(\Omega))$ and, therefore, $\Lambda(\mathcal E_R)$ is compact in $L^2(Q)$ by Aubin-Lions' Theorem (\cite{simon1986compact}).

Notice that, for any $w\in \mathcal E_R$, we have at least one control $v$ such that the corresponding solution $y$ belongs to $\mathcal E_R$. Hence, for the sequence $\left\{w_k\right\}$ we can find a sequence of controls $\left\{v_k\right\}$ such that the corresponding solutions $\left\{y_k\right\}$ is in $L^2(Q)$. Let $w_k\to w$ in $\mathcal E_R$ and $y_k\in \Lambda(w_k)$, $y_k\to y$ in $L^2(Q)$. We want to show that $y\in\Lambda(w)$. By the regularity of the solutions and \eqref{estimate} it follows (selecting a subsequence if necessary) that
\begin{align*}
	& v_k\rightharpoonup v \textrm{ weakly in } L^2(\mathcal O\times(0,T)),
	\\
	& y_k\rightharpoonup y \textrm{ weakly in } L^2(0,T;H^1_0(\Omega))\cap H^1(0,T;H^{-1}(\Omega)),
	\\
	& y_k\to y \textrm{ strongly in } L^2(Q). 
\end{align*}
Then we obtain $y\in L^2(Q)$ and, letting $k\to +\infty$ in the system 
\begin{align*}
	\begin{cases}
		\D (y_k)_t - \Delta y_k + \int_\Omega K(x,\theta,t)w_k(\theta,t)\,d\theta = v_k\mathbf{1}_{\mathcal O}, & (x,t)\in Q
		\\
		y_k = 0, & (x,t)\in\Sigma
		\\
		y_k(x,0) = y_0(x), & x\in\Omega.
	\end{cases}
\end{align*} 
we can conclude that the couple $(y,v)$ satisfies \eqref{control_w}, i.e. $\Lambda(w) = y$. Thus the map $\Lambda$ is upper hemicontinuous.

Therefore, all the assumptions of Kakutani's fixed point theorem are fulfilled and we infer that there is at least one $y\in \mathcal E_R$ such that $y=\Lambda(y)$. By the definition of $\Lambda$, this implies that there exists at least one pair $(u,y)$ satisfying the
conditions of Theorem \ref{control_thm2}. The fact that $y(x,T) = 0$ in $\Omega$ comes from the definition of the space $\mathcal E$ and the weight function $\beta$. Hence, our assertion is proved.	
\end{proof}
\begin{remark}
As we were anticipating in Section \ref{intro_sec}, even though the approach of Theorem \ref{control_thm2} has the advantage of not requiring any decay in the kernel as $t$ goes to $0^+$, due to the nature of the employed fixed point argument we lose the uniqueness of the control function $v$. In particular, we are not able to identify a distinguished control with a constructive procedure. As a consequence of this fact, we also lose information about the cost of null controllability. Indeed, when a control can be computed by minimizing, for instance, the functional \eqref{functional_J}, the null controllability cost is related to the square root of the constant in the observability inequality \eqref{e12s1}. On the other hand, since the proof of Theorem \ref{control_thm2} does not requires any observability, in that case we cannot recover any explicit expression for $\mathcal C(y_0)$.
\end{remark}

\section{Extension to semilinear problems}\label{nl_sec}

The approach of Theorem \ref{control_thm} may be combined with the methodology presented in \cite{fabre1995approximate,fernandez1999approximate,fernandez1997null,zuazua1991exact} to deduce similar controllability results for the semilinear heat equation with globally Lipschitz nonlinearity  
\begin{align}\label{heat_nl}
\begin{cases}
\D y_t - \Delta y + \int_\Omega K(x,\theta,t)y(\theta,t)\,d\theta = f(y) + v\mathbf{1}_{\mathcal O}, & (x,t)\in Q
\\
y = 0, & (x,t)\in\Sigma
\\
y(x,0) = y_0(x), & x\in\Omega.
\end{cases}
\end{align}
In more detail, it is possible to prove the following result.
\begin{theorem}\label{control_thm_nl}
Assume $f\in C^1(\R)$ is globally Lipschitz with $f(0)=0$. Then, given any $y_0\in L^2(\Omega)$ and $T>0$ there exists a control function $v\in L^2(\mathcal O\times(0,T))$ such that the solution to \eqref{heat_nl} satisfies $y(x,T)=0$.
\end{theorem}

\begin{proof}[Sketch of the proof]
The proof of this result is by now standard and it uses well-known results on the controllability of nonlinear systems (\cite{fabre1995approximate,fernandez1999approximate,fernandez1997null,zuazua1991exact}). For the sake of completeness, we sketch below the main steps. 	
	
Since $f\in C^1(\R)$, we can introduce the function $g:\R\to\R$ defined as
\begin{align*}
	g(s):=\begin{cases}
		\displaystyle \frac{f(s)}{s}, & \textrm{ if } s\neq 0
		\\[6pt]
		f'(0), & \textrm{ if } s=0.
	\end{cases}	
\end{align*}
Then, for all $\eta\in L^2(Q)$ we can consider the following linearized version of \eqref{heat_nl}
\begin{align}\label{heat_lin}
	\begin{cases}
		\D y_t - \Delta y + \int_\Omega K(x,\theta,t)y(\theta,t)\,d\theta = g(\eta)y + v\mathbf{1}_{\mathcal O}, & (x,t)\in Q
		\\
		y = 0, & (x,t)\in\Sigma
		\\
		y(x,0) = y_0(x), & x\in\Omega.
	\end{cases}
\end{align}
	
The continuity of $f$ and the density of $C_0^\infty(Q)$ in $L^2(Q)$ allows to see that $g(\eta)\in L^\infty(Q)$ for all $\eta\in L^2(Q)$. Therefore, arguing as in the proof of Theorem \ref{th2s1} we can obtain the following observability estimate 
\begin{align}\label{obs_nl}
	\norm{\varphi(x,0)}{L^2(\Omega)}\leq C_1\exp\left[C_2\left(1 + \frac 1T + T\norm{g}{\infty} + \norm{g}{\infty}^{\frac 23} + \mathcal K^{\frac 23}\right)\right] \intd_{\mathcal O\times(0,T)} |\varphi|^2\,dxdt, 
\end{align}
where $\varphi$ is the solution to the adjoint system associated to \eqref{heat_lin}. This in particular implies that \eqref{heat_lin} is null-controllable in time $T>0$ with a control $v_\eta\in L^2(\mathcal O\times(0,T))$ satisfying
\begin{align*}
	\norm{v_\eta}{L^2(\mathcal O\times(0,T)}\leq\sqrt{\mathcal C}\norm{y_0}{L^2(\Omega)}, \;\;\; \forall\,\eta\in L^2(Q),
\end{align*}
where with $\mathcal C$ we indicate the constant in the inequality \eqref{obs_nl}. 
	
Consider the map $\Lambda: L^2(Q)\to L^2(Q)$ defined by $\Lambda\eta = y_\eta$, where $y_\eta$ is the solution to \eqref{heat_lin} corresponding to the control $v_\eta$. By means of \eqref{y_reg}, we deduce that $\Lambda$ maps $L^2(Q)$ into a bounded set of $L^2(0,T,H^1_0(\Omega))\cap H^1(0,T,H^{-1}(\Omega))$. This space being compactly embedded in $L^2(Q)$, there exists a fixed compact set $W$ such that $\Lambda(L^2(Q))\subset W$. Moreover, it can be readily verified that $\Lambda$ is also continuous from $L^2(Q)$ into $L^2(Q)$. In view of that, applying the Schauder fixed point theorem and proceeding like in the conclusion of the proof of Theorem \ref{control_thm2}, the result follows immediately.
\end{proof}

We conclude this section by mentioning that, in the same spirit of Theorem \ref{control_thm_nl}, it is possible to address also more general versions of \eqref{heat_nl} in which the nonlinearity is included in the non-local term, that is, problems in the form
\begin{align}\label{heat_nl_gen}
	\begin{cases}
		\D y_t - \Delta y + \int_\Omega K(x,\theta,t)f\big[y(\theta,t)\big]\,d\theta = v\mathbf{1}_{\mathcal O}, & (x,t)\in Q
		\\
		y = 0, & (x,t)\in\Sigma
		\\
		y(x,0) = y_0(x), & x\in\Omega,
	\end{cases}
\end{align}
where $f$ has the same regularity properties as in Theorem \ref{control_thm_nl}. 

For doing that, we just need to prove a Carleman estimate for the linearized adjoint system corresponding to \eqref{heat_nl_gen}, which reads as
\begin{align}\label{heat_nl_gen_adj}
	\begin{cases}
		\D -\varphi_t - \Delta \varphi + g(\eta)\int_\Omega K(\theta,x,t)\varphi(\theta,t)\,d\theta = 0, & (x,t)\in Q
		\\
		\varphi = 0, & (x,t)\in\Sigma
		\\
		\varphi(x,0) = \varphi_T(x), & x\in\Omega,
	\end{cases}
\end{align}
to obtain from there an observability inequality for \eqref{heat_nl_gen_adj} and conclude by following the same argument as in Theorem \ref{control_thm_nl}.

The proof of such Carleman inequality is a straightforward adaptation of Proposition \ref{carleman_phi_prop}. Indeed, it is sufficient to notice that, in this case, \eqref{carleman_prel} becomes
\begin{align*}
	\mathcal{I}(\varphi) &\leq C \left[s^3\lambda^4\intd_{\mathcal O\times(0,T)} e^{-2s\alpha}\xi^3 |\varphi|^2\,dx\,dt + \intd_Q e^{-2s\alpha}\left|g(\eta)\int_{\Omega} K(\theta,x,t)\varphi(\theta,t)\,d\theta\,\right|^2\,dx\,dt\right]
	\\
	&\leq C \left[s^3\lambda^4\intd_{\mathcal O\times(0,T)} e^{-2s\alpha}\xi^3 |\varphi|^2\,dx\,dt + \norm{g}{\infty}^2\intd_Q e^{-2s\alpha}\left|\int_{\Omega} K(\theta,x,t)\varphi(\theta,t)\,d\theta\,\right|^2\,dx\,dt\right],
\end{align*}
since $g(\eta)\in L^\infty(Q)$. From here, the remaining of the proof is the same as we did before, with the only change that we now have to choose 
\begin{align*}
	s\geq \varrho_3\left[T+T^2+\left(\mathcal K\norm{g}{\infty}\right)^{\frac 23}T^2\right],
\end{align*}
with $\varrho_3$ a positive constant only depending on $\Omega$ and $\mathcal O$. In view of that, the observability estimate that we obtain is in the form
\begin{align}\label{obs_nl2}
	\norm{\varphi(x,0)}{L^2(\Omega)}\leq C_1\exp\left[C_2\left(1 + \frac 1T + T\left(\mathcal K\norm{g}{\infty}\right)^{\frac 23}\right)\right] \intd_{\mathcal O\times(0,T)} |\varphi|^2\,dxdt.
\end{align}

From \eqref{obs_nl2}, the null controllability in time $T>0$ for \eqref{heat_nl_gen_adj} follows immediately by means of a classical argument.

\section{On the necessity of hypothesis \eqref{K_est_gen}}\label{comments_sec}

As anticipated in Remark \ref{rem_lambda}, hypotheses on the kernel $K$ more general that just being bounded are necessary. Indeed, not imposing any assumption further than $\norm{K}{\infty}<+\infty$ may lead to the failure of the unique continuation of the solutions to the adjoint system. In fact, in absence of additional conditions on the kernel, it is possible to provide counterexamples where unique continuation (which is essential for obtaining controllability results) fails.

In what follows, we present one which has been proposed by P. Gerard. Let us consider the following one-dimensional situation. Let $u\in C_0^\infty(0,1)$ be a function verifying $u(x)=0$ for $x\in(a,b)\subset(0,1)$ but not identically zero on the whole interval $(0,1)$ (see Figure \ref{figure_u}). 
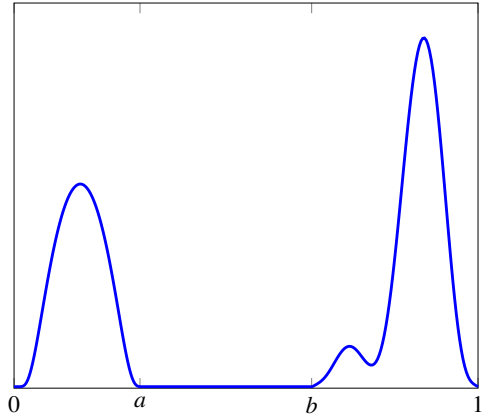
\begin{figure}[h]
	\centering 
	\pgfplotstableread{./data.txt}{\datos}
	\begin{tikzpicture}[scale=0.9]
	\begin{axis}[xmin = -4, xmax = 3, ymin=-0.01,xtick={-4,-2.1,0.49,3}, ytick=\empty, xticklabels={$0$,$a$,$b$,$1$}]
	\addplot [solid,very thick, color=blue] table[x=0,y=1]{\datos};
	\end{axis}
	\end{tikzpicture}\caption{Example of a function $u(x)$ verifying: (i) $u\in C_0^\infty(0,1)$, (ii) $u(x)=0$ for $x\in(a,b)$ and (iii) $u\not\equiv 0$ in $(0,1).$}\label{figure_u}
\end{figure}	


Since $u$ is in particular a $L^2(0,1)$ function, we can write it in the form
\begin{align*}
	u(x)=\sum_{k\geq 1} c_k\phi_k(x),
\end{align*}
with $\phi_k(x)=\sqrt{2}\sin(k\pi x)$ and $c_k=\langle u,\phi_k\rangle_{L^2(0,1)}$. Moreover, for $0<\lambda<\pi$, we have 
\begin{align*}
	\sum_{k\geq 1}\left(k^2\pi^2-\lambda^2\right)c_k^2 >0
\end{align*}
and, up to a change of variables of the type $u\mapsto\sigma u$, $\sigma>0$, we can assume 
\begin{align*}
	\sum_{k\geq 1}\left(k^2\pi^2-\lambda^2\right)c_k^2 =1.
\end{align*}	
Define
\begin{align*}
	p(x) = \sum_{k\geq 1} \left(k^2\pi^2-\lambda^2\right)c_k\phi_k(x).
\end{align*}

It can be readily checked that $-u_{xx}-\lambda^2 u = p$ is verified in the sense of distributions. Hence $p\in C_0^\infty(0,1)$ with $p(x)=0$ in $(a,b)$, since $u$ has these properties. Moreover, by definition of $p$ we have
\begin{align*}
	\int_0^1 pu\,dx = 1.
\end{align*}
Therefore, $u$ satisfies the nonlocal elliptic problem
\begin{align} 
	\begin{cases}\label{nonlocal_elliptic}
		\D -u_{xx}+\int_0^1 K(x,\theta)u(\theta)\,d\theta = \lambda^2 u, & x\in(0,1) 
		\\
		u(0)=u(1)=0
	\end{cases}
\end{align}
with $K(x,\theta)=p(x)p(\theta)$. Furthermore, by assumption $u(x)=0$ for $x\in(a,b)\subset(0,1)$ but $u\not\equiv 0$ elsewhere. 

In other words, we constructed an example of a function which is solution to \eqref{nonlocal_elliptic} and does not satisfies unique continuation. In addition, this fact can be extended to the parabolic case by means of classical techniques, thus implying the failure of any controllability property for our original equation.

To avoid these kind of situations, and to ensure the validity of unique continuation properties, in some recent works the following solutions have been proposed:
\begin{itemize}
	\item analyticity assumptions on the kernel, that allow to see it as a compact perturbation of the heat operator (\cite{fernandez2016null});
	\item kernels in the form $k(x,\theta)=\alpha(x)\beta(\theta)$, with $\alpha$ not vanishing in any subinterval of $(0,1)$, which preserve the unique continuation for the solutions (\cite[Lemma 2.15]{micu2018local}).
\end{itemize}

In the present work, thanks to \eqref{K_est_gen} and by taking advantage of the weight function $\alpha$, we are able to absorb the nonlocal term in the terms on the left-hand side of \eqref{carleman_prel}, then obtaining a nice Carleman estimate for the solution to \eqref{e11s1}. This, of course, gives us the unique continuation and allow us to prove an observability inequality, therefore a controllability result.
In conclusion, from the discussion above it is clear that kernels that are merely bounded cannot be handled when dealing with controllability problems for nonlocal equations of the type of \eqref{e1s1}. Instead, some additional condition has to be imposed. In this paper, we propose to consider kernels depending also on the time variables, and to exploit the structure of the weights in the Carleman estimate. Nonetheless, we do not know whether our approach is the best possible or if, instead, sharper results can be obtained.

\section*{Acknowledgements} 
This project has received funding from the European Research Council (ERC) under the European Union's Horizon 2020 research and innovation programme (grant agreement No. 694126-DyCon) and from the MTM2017-92996 grant of MINECO (Spain). The work of the first author was partially supported by the Grants MTM2014-52347 of MINECO (Spain) and FA9550-18-1-0242 of AFOSR (U.S.). 

The authors wish to thank Prof. Enrique Zuazua (Universidad Aut\'onoma de Madrid and DeustoTech) for interesting discussions on the topic of this paper. Thanks also to Cristhian Montoya (Universidad Nacional Aut\'onoma de M\'exico) for his valuable comments that helped to improve this manuscript. Finally, a special thanks goes to Prof. Patrick G\'erard (Universit\'e Paris-Sud) for suggesting the nice counterexample that helps confirming the necessity of certain hypothesis of our work.

\bibliography{biblio}

\begin{thebibliography}{10}

\bibitem{fabre1995approximate}
C.~Fabre, J.-P. Puel, and E.~Zuazua.
\newblock Approximate controllability of the semilinear heat equation.
\newblock {\em Proc. Roy. Soc. Edinburgh Sec. A}, 125(1):31--61, 1995.

\bibitem{fernandez1999approximate}
L.~A. Fern{\'a}ndez and E.~Zuazua.
\newblock Approximate controllability for the semilinear heat equation
  involving gradient terms.
\newblock {\em J. Optim. Theor. Appl.}, 101(2):307--328, 1999.

\bibitem{fernandez1997null}
E.~Fernandez-Cara.
\newblock Null controllability of the semilinear heat equation.
\newblock {\em ESAIM: Control Optim. Calc. Var.}, 2:87--103, 1997.

\bibitem{fernandez2006global}
E.~Fern{\'a}ndez-Cara and S.~Guerrero.
\newblock Global {C}arleman inequalities for parabolic systems and applications
  to controllability.
\newblock {\em SIAM J. Control Optim.}, 45(4):1395--1446, 2006.

\bibitem{fernandez2004local}
E.~Fern{\'a}ndez-Cara, S.~Guerrero, O.~Y. Imanuvilov, and J.-P. Puel.
\newblock Local exact controllability of the {N}avier-{S}tokes system.
\newblock {\em J. Math. Pures Appl.}, 83(12):1501--1542, 2004.

\bibitem{fernandez2016null}
E.~Fern{\'a}ndez-Cara, Q.~L{\"u}, and E.~Zuazua.
\newblock Null controllability of linear heat and wave equations with nonlocal
  spatial terms.
\newblock {\em SIAM J. Control Optim.}, 54(4):2009--2019, 2016.

\bibitem{fernandez2000cost}
E.~Fern{\'a}ndez-Cara and E.~Zuazua.
\newblock The cost of approximate controllability for heat equations: the
  linear case.
\newblock {\em Adv. Differ. Eq.}, 5(4-6):465--514, 2000.

\bibitem{fursikov1996controllability}
A.~V. Fursikov and O.~Y. Imanuvilov.
\newblock {\em Controllability of evolution equations}.

\bibitem{gueye2013insensitizing}
M.~Gueye.
\newblock Insensitizing controls for the {N}avier-{S}tokes equations.
\newblock {\em Ann. Inst. H. Poincar{\'e} Anal. Non lin{\'e}aire},
  30(5):825--844, 2013.

\bibitem{kakutani1941generalization}
S.~Kakutani et~al.
\newblock A generalization of {B}rouwer’s fixed point theorem.
\newblock {\em Duke Math. J.}, 8(3):457--459, 1941.

\bibitem{lissy2018internal}
P.~Lissy and E.~Zuazua.
\newblock Internal controllability for parabolic systems involving analytic
  non-local terms.
\newblock {\em Chin. Ann. Math. Ser. B}, 39(2):281--296, 2018.

\bibitem{lorenzi2011two}
A.~Lorenzi.
\newblock Two severely ill-posed linear parabolic problems.
\newblock In {\em AIP Conference Proceedings}, volume 1329, pages 150--169.
  AIP, 2011.

\bibitem{micu2018local}
S.~Micu and T.~Takahashi.
\newblock Local controllability to stationary trajectories of a {B}urgers
  equation with nonlocal viscosity.
\newblock {\em J. Differential Equations}, 264(5):3664--3703, 2018.

\bibitem{montoya2017robust}
C.~Montoya and L.~de~Teresa.
\newblock Robust {S}tackelberg controllability for the {N}avier-{S}tokes
  equations.
\newblock {\em arXiv preprint arXiv:1708.04648}, 2017.

\bibitem{okubo2013diffusion}
A.~Okubo and S.~A. Levin.
\newblock {\em Diffusion and ecological problems: modern perspectives},
  volume~14.
\newblock Springer Science \& Business Media, 2013.

\bibitem{simon1986compact}
J.~Simon.
\newblock Compact sets in the space ${L}^p(0,{T};{B})$.
\newblock {\em Ann. Mat. Pura Appl.}, 146(1):65--96, 1986.

\bibitem{tao2016null}
Q.~Tao and H.~Gao.
\newblock On the null controllability of heat equation with memory.
\newblock {\em J. Math. Anal. Appl.}, 440(1):1--13, 2016.

\bibitem{zuazua1991exact}
E.~Zuazua.
\newblock Exact boundary controllability for the semilinear wave equation.
\newblock {\em Nonlinear partial differential equations and their
  applications}, 10:357--391, 1991.

\end{thebibliography}

\end{document}